\documentclass[]{amsart}
\usepackage{mathrsfs}
\usepackage{graphicx}

\usepackage{amssymb}
\usepackage{amsmath}
\usepackage{amsthm}
\usepackage{galois}

\allowdisplaybreaks                 
\usepackage{float}                 
\restylefloat{figure}              
\usepackage{caption2}              
\usepackage{array}              
\usepackage{slashbox}              
\newtheorem{thm}{Theorem}[section]
\newtheorem{cor}[thm]{Corollary}
\newtheorem{lem}[thm]{Lemma}
\newtheorem{prop}[thm]{Proposition}
\theoremstyle{definition}
\newtheorem{defn}[thm]{Definition}

\newtheorem*{question}{Question}
\theoremstyle{remark}
\newtheorem{rem}[thm]{Remark}
\numberwithin{equation}{section}

\newcommand{\Z}{\mathbb Z}
\newcommand{\Q}{\mathbb Q}

\newcommand{\fix}{\mathrm{Fix}}
\newcommand{\ind}{\mathrm{ind}}

\newcommand{\fpc}{\mathrm{Fpc}}
\newcommand{\F}{\mathbf{F}}      

\newcommand{\trace}{\mathrm{Trace}}
\newcommand{\B}{\mathcal{B}}
\newcommand{\im}{\mathrm{Image}}
\def\a{\alpha}
\def\b{\beta}

\begin{document}

\title{Bounds for fixed points on products of hyperbolic surfaces}
\author{Qiang Zhang, Xuezhi Zhao}
\address{School of Mathematics and Statistics, Xi'an Jiaotong University,
Xi'an 710049, China}
\email{zhangq.math@mail.xjtu.edu.cn}
\address{Department of Mathematics, Capital Normal University, Beijing 100048, China}
\email{zhaoxve@mail.cnu.edu.cn}

\thanks{The first author is partially supported by NSFC (No. 11771345), and the second author is partially supported by NSFC (No. 11431009 and No. 11661131004).}

\subjclass[2010]{55M20, 55N10, 32Q45}

\keywords{Nielsen numbers, Lefschetz numbers, indices, bounds, hyperbolic surfaces, products}

\dedicatory{Dedicated to Professor Boju Jiang on his 80th birthday}
\begin{abstract}
For the product $S_1\times S_2$ of any two connected compact hyperbolic surfaces $S_1$ and $S_2$, we give a finite bound $\mathcal{B}$ such that for any self-homeomorphism $f$ of $S_1\times S_2$ and any fixed point class $\F$ of $f$, the index $|\ind(f, \F)|\leq \mathcal{B}$, which is an affirmative answer for a special case of a question asked by Boju Jiang. Moreover, we also give bounds for the Lefschetz number $L(f)$ and the Nielsen number $N(f)$ of the homeomorphism $f$.
\end{abstract}
\maketitle


\section{Introduction}

Fixed point theory studies fixed points of a selfmap $f$ of a space $X$. Nielsen fixed point theory, in particular, is concerned with
the properties of the fixed point set
$$\fix f:=\{x\in X|f(x)=x\}$$
that are invariant under homotopy of the map $f$ (see \cite{fp1} for an introduction).

The fixed point set $\fix f$ splits into a disjoint union of \emph{fixed point classes}: two fixed
points $a$ and $a'$ are in the same class if and only if there is a lifting $\tilde f: \widetilde X\to \widetilde X$ of $f$ such that $a, a'\in p(\fix \tilde f)$, where $p:\widetilde X\to X$ is the universal cover. Let $\fpc(f)$ denote the set of all the fixed point classes of $f$. For each fixed point class $\F\in \fpc(f)$, a homotopy invariant \emph{index} $\ind(f,\F)\in \Z$ is defined. A fixed point class is \emph{essential} if its index is non-zero. The number of essential fixed point classes of $f$ is called the $Nielsen$ $number$ of $f$, denoted by $N(f)$. The famous Lefschetz fixed point theorem says that the sum of the indices of the fixed points of $f$ is equal to the $Lefschetz$ $number$ $L(f)$, which is defined as
$$L(f):=\sum_q(-1)^q\trace (f_*: H_q(X;\Q)\to H_q(X;\Q)).$$

A compact polyhedron $X$ is said to have the \emph{Bounded Index Property (BIP)} if there is an integer $\B>0$ such that for any map $f: X\rightarrow X$ and any fixed point class $\F$ of $f$, the index $|\ind(f,\F)|\leq \B$. $X$ has the \emph{Bounded Index Property for Homeomorphisms (BIPH)} if there is such a bound for all homeomorphisms $f:X\rightarrow X$.

In \cite{fp2}, B. Jiang showed that graphs and surfaces with negative Euler characteristics have BIP (see \cite{JWZ} for an enhanced version):

\begin{thm}[Jiang, \cite{fp2}]\label{fpc for surf.}
Suppose $X$ is a connected compact surface with Euler characteristic $\chi(X) < 0$, and suppose $f : X \to X$ is a selfmap.
Then the indices of the Nielsen fixed point classes of $f$ are bounded:

$\mathrm{(A)}$ $\ind(\F)\leq 1$ for every fixed point class
$\F$ of $f$;

$\mathrm{(B)}$ almost every fixed point class $\F$ of $f$ has index $\geq -1$, in the sense
that
$$\sum_{\ind(\F)<-1}\{\ind(\F)+1\}\geq 2\chi(X),$$
where the sum is taken over all fixed point classes $\F$ with
$\ind(\F)<-1$;

$\mathrm{(C)}$ $|L(f)-\chi(X)|\leq N(f)-\chi(X),$ where $L(f)$ and $N(f)$ are the Lefschetz number
and the Nielsen number of $f$ respectively.

\end{thm}

Moreover, he asked the following question (see \cite[Qusetion 3]{fp2}): Does every compact aspherical polyhedron $X$ (i.e. $\pi_i(X)=0$ for all $i>1$) have BIP or BIPH?

In \cite{Mc}, C. McCord showed that infrasolvmanifolds (manifolds which admit a finite cover by a compact solvmanifold) have BIP. In \cite{JW}, B. Jiang and S. Wang showed that geometric 3-manifolds have BIPH for orientation-preserving self-homeomorphisms: the index of each essential fixed point class is $\pm 1$. In \cite{Z}, the first author showed that orientable compact Seifert 3-manifolds with hyperbolic orbifolds have BIPH. Later in \cite{Z2, Z3}, the first author showed that compact hyperbolic $n$-manifolds (not necessarily orientable) also have BIPH.

In this note, we consider the product of two connected compact surfaces with negative Euler characteristics, and show it has BIPH. Such a surface is also said to be a hyperbolic one, because it always admits a Riemannian metric of constant curvature $-1$.

The main result of this note is the following

\begin{thm}\label{main thm}
Suppose $f: S_1\times S_2\to S_1\times S_2$ is a homeomorphism, where $S_1$ and $S_2$ are two connected compact surfaces with Euler characteristics $\chi_1:=\chi(S_1)\leq \chi_2:=\chi(S_2)<0.$ Then the indices of the Nielsen fixed point classes of $f$ are bounded:

$\mathrm{(A)}$ For every fixed point class $\F$ of $f$, we have
$$2\chi_1-1\leq \ind(f,\F)\leq (2\chi_1-1)(2\chi_2-1);$$

$\mathrm{(B)}$ Let $L(f)$ and $N(f)$ be the Lefschetz number
and the Nielsen number of $f$ respectively. Then
$$|L(f)-2\chi_1\chi_2|\leq (1-2\chi_1)N(f)+2(\chi_1\chi_2-\chi_1).$$
\end{thm}

We also consider a special case of selfmaps of $S_1\times S_2$, and give some bounds (see Proposition \ref{fiber case}) parallel to Theorem \ref{main thm} on the index, the Lefschetz number
and the Nielsen number in Section 2. In Section 3, we give some bounds for the fixed points on the alternating homeomorphisms, see Proposition \ref{alter case}. Finally, in Section 4, we complete the proof of Theorem \ref{main thm}, and give an open question.

\section{Fiber-preserving maps}

In this section, let $S_1, S_2$ be two connected compact surfaces, and $\chi_i$ the Euler characteristic of $S_i$, $\chi_1\leq \chi_2<0$.

\begin{defn} A selfmap $f: S_1\times S_2\to S_1\times S_2$ is called a \emph{fiber-preserving map}, if
$$f=f_1\times f_2: S_1\times S_2\to S_1\times S_2,\quad (a,b)\mapsto (f_1(a),f_2(b)),$$
where $f_i$ is a selfmap of $S_i(1=1,2)$.
\end{defn}

\begin{lem}\label{product}
If $f: S_1\times S_2\to S_1\times S_2$ is a fiber-preserving map, then $\fix f=\fix f_1\times \fix f_2,$ and each fixed point class $\F\in \fpc(f)$ splits into a product of some fixed point classes of $f_i$, i.e.,
$$\F=\F_1\times \F_2, \quad\quad \ind(f,\F)=\ind(f_1,\F_1)\cdot \ind(f_2, \F_2),$$
where $\F_i\in \fpc (f_i)$ is a fixed point class of $f_i$ for $i=1,2$. Moreover,
$$L(f)=L(f_1)\cdot L(f_2),\quad N(f)=N(f_1)\cdot N(f_2).$$
\end{lem}

\begin{proof}
Note that
$$\fix f=\{(a,b)|f_1(a)=a, f_2(b)=b\}=\fix f_1\times \fix f_2.$$
Suppose $(a,b)$ and $(a',b')$ are two fixed points in the same fixed point class $\F\in \fpc(f)$, then there is a lifting $\tilde f_1$ of $f_1$ and a lifting $\tilde f_2$ of $f_2$ such that $$(a,b),(a',b')\in (p_1\times p_2)(\fix(\tilde f_1\times \tilde f_2))$$
where $p_i: \widetilde S_i\to S_i$ is the universal cover for $i=1,2$. Hence $a, a'\in p_1(\fix \tilde f_1)$ and $b, b'\in p_2(\fix \tilde f_2)$, namely, $a, a'$ are in the same fixed point class $\F_1$ of $f_1$ and $b,b'$ in the same fixed point class $\F_2$ of $f_2$. Conversely, if $a, a'$ are in the same fixed point class $\F_1$ of $f_1$ and $b,b'$ in the same fixed point class $\F_2$ of $f_2$, then there is a lifting $\tilde f_1$ of $f_1$ and a lifting $\tilde f_2$ of $f_2$ such that $a, a'\in p_1(\fix \tilde f_1)$ and  $b, b'\in p_2(\fix \tilde f_2)$. It implies that $(a,b),(a',b')\in (p_1\times p_2)(\fix(\tilde f_1\times \tilde f_2))$. Since $\tilde f_1\times \tilde f_2$ is a lifting of $f=f_1\times f_2$, we obtain that $(a,b)$ and $(a',b')$ are in the same fixed point class $\F$ of $f$. Therefore, $\F=\F_1\times \F_2$, and hence $\ind(f,\F)=\ind(f_1,\F_1)\cdot \ind(f_2, \F_2)$ by the multiplicativity of the index.

Since $f=f_1\times f_2$ is a fiber-preserving map, by ``Product formula for the Lefschetz number" \cite[pp.85, Theorem 3.2]{fp1} and ``Product formula for the Nielsen number \cite[pp.88, Theorem 4.1]{fp1}, we have $L(f)=L(f_1)\cdot L(f_2)$ and $N(f)=N(f_1)\cdot N(f_2).$
\end{proof}

Now we give the bounds for indices, Lefschetz numbers and Nielsen numbers of fiber-preserving maps.

\begin{prop}\label{fiber case} If $f: S_1\times S_2\to S_1\times S_2$ is a fiber-preserving map, then

$\mathrm{(A)}$ For every fixed point class $\F$ of $f$, we have
$$2\chi_1-1\leq \ind(f,\F)\leq (2\chi_1-1)(2\chi_2-1);$$

$\mathrm{(B)}$ Let $L(f)$ and $N(f)$ be the Lefschetz number
and the Nielsen number of $f$ respectively. Then
$$|L(f)-2\chi_1\chi_2|\leq (1-2\chi_1)N(f)+2(\chi_1\chi_2-\chi_1).$$
\end{prop}

\begin{proof}
(A) Bounds for indices

According to Theorem \ref{fpc for surf.}, for every fixed point class $\F_i$ of $f_i$, we have
$$1\geq \ind(f_i, \F_i)\geq 2\chi_i-1\geq 2\chi_1-1.$$
So by Lemma \ref{product},
$$\ind(f,\F)=\ind(f_1,\F_1)\cdot \ind(f_2, \F_2)\geq 2\chi_1-1,$$
$$\ind(f,\F)=\ind(f_1,\F_1)\cdot \ind(f_2, \F_2)\leq (2\chi_1-1)(2\chi_2-1),$$
i.e., for every fixed point class $\F$ of $f$, conclusion (A) holds.

(B) Bounds for Lefschetz \& Nielsen numbers

If either $N(f_1)$ or $N(f_2)$ is zero, then $N(f)=0$, and $L(f)=0$ according to the Lefschetz fixed point theorem. Since $\chi_1\leq \chi_2<0$, conclusion (B) clearly holds in this case.

Now we suppose $N(f_1)\geq 1$ and $N(f_2)\geq 1$.

According to Theorem \ref{fpc for surf.}, we have
$$N(f_i)\geq L(f_i)\geq -N(f_i)+2\chi_i,~~ i=1,2.$$
Then
\begin{eqnarray}
L(f)&=&L(f_1)\cdot L(f_2) \nonumber\\
&\leq& (N(f_1)-2\chi_1)(N(f_2)-2\chi_2)\nonumber\\
&=& N(f_1)N(f_2)+4\chi_1\chi_2-2(\chi_1N(f_2)+\chi_2N(f_1)) \nonumber\\
&\leq& N(f)+4\chi_1\chi_2 -2\chi_1(N(f_1)+N(f_2)) \nonumber\\
&\leq& N(f)+4\chi_1\chi_2 -2\chi_1(N(f_1)N(f_2)+1) \nonumber\\
&=& (1-2\chi_1)N(f)+4\chi_1\chi_2-2\chi_1\nonumber
\end{eqnarray}
and
\begin{eqnarray}
L(f)&\geq& min\{N(f_1)(-N(f_2)+2\chi_2),~(-N(f_1)+2\chi_1)N(f_2)\}\nonumber\\
&\geq& -N(f) +2\chi_1(N(f_1)+N(f_2)) \nonumber\\
&\geq& -N(f) +2\chi_1(N(f_1)N(f_2)+1) \nonumber\\
&=& -(1-2\chi_1)N(f) +2\chi_1\nonumber
\end{eqnarray}
Hence
$$|L(f)-2\chi_1\chi_2|\leq (1-2\chi_1)N(f)+2(\chi_1\chi_2-\chi_1),$$
i.e., conclusion (B) holds.
\end{proof}

\section{Alternating homeomorphisms}

In this section, let $S_1= S_2$ be two copies of a connected compact hyperbolic surface $S$, and hence, their Euler characteristics $\chi_1=\chi_2=\chi(S)<0$.

\begin{defn} A self-homeomorphism $f: S_1\times S_2\to S_1\times S_2$ is called an \emph{alternating homeomorphism}, if
$$f=\tau\comp (f_1\times f_2): S_1\times S_2\to S_1\times S_2, \quad (a,b)\mapsto (f_2(b),f_1(a)),$$
where $f_1$, $f_2$ are two self-homeomorphisms of $S$, and $\tau$ is a transposition.
\end{defn}

\begin{lem}\label{transversal}
Two self-homeomorphisms $f_1$ and $f_2$ are isotopic to $g_1$ and $g_2$ such that the graph of corresponding alternating
homeomorphism $g=\tau\comp (g_1\times g_2): S_1\times S_2\to S_1\times S_2$ is transversal with the diagonal in $S_1\times S_2$. Moreover, the distance $d(f_1, g_1)$ and $d(f_2, g_2)$ can be chosen arbitrary small.
\end{lem}

\begin{proof}
Note that for a compact hyperbolic surface, every homeomorphism is isotopic to a diffeomorphism. Since $S_1=S=S_2$ is compact, we may choose an atlas $\{(U_i, \psi_i)\mid i=1,2,\ldots n\}$ with finite elements. For each fixed point $(a,b)$ of $f$, there is an open neighborhood $P$ of $a$ and an open neighborhood $Q$ of $b$ such that $\bar P\cup f_2(\bar Q)\subset U_i$ and $\bar Q\cup f_1(\bar P)\subset U_j$ for some $i$ and $j$. The compactness of the fixed point set of $f$ implies that we can chose finitely many $P_k\times Q_k$, which is a neighborhood of some fixed point of $f$, such that $\fix f\subset \cup_{k=1}^m P_k\times Q_k$.

We shall modify $f_1$ and $f_2$ on $P_k\times Q_k$ inductively so that $f$ has its desired transversality. Begin with $P_1\times Q_1$. We have that  $\bar P_1\cup f_2(\bar Q_1)\subset U_i$ and $\bar Q_1\cup f_1(\bar P_1)\subset U_j$ for some $i$ and $j$. Then $(P_1\times Q_1, \psi_i|_{P_1}\times \psi_j|_{Q_1})$ is a chart of $S_1\times S_2$ at some fixed point $(a,b)$ of $f$, and $(U_i\times U_j\times U_i\times U_j,\psi_i\times \psi_j\times \psi_i\times \psi_j)$ is a chart of $(S^1\times S^2)^2$. Under these charts, the diagonal map and the graph map of $f$ are respectively given by
$$(u_1, u_2, u_3, u_4)\mapsto (u_1, u_2, u_3, u_4, u_1, u_2, u_3, u_4)$$
and
$$(u_1, u_2, u_3, u_4)\mapsto (u_1, u_2, u_3, u_4, f_{21}(u_3, u_4), f_{22}(u_3, u_4), f_{11}(u_1, u_2), f_{12}(u_1, u_2)),$$
where $(f_{11}, f_{12}) = \psi_j\comp f_1\comp \psi_i^{-1}|_{\psi_i(P_1)}$ and $(f_{21}, f_{22}) = \psi_i\comp f_2\comp\psi_j^{-1}|_{\psi_j(Q_1)}$.
Hence, their differentials are respectively
$$
\left(
  \begin{array}{cccc}
    1 & 0 & 0 & 0 \\
    0 & 1 & 0 & 0 \\
    0 & 0 & 1 & 0 \\
    0 & 0 & 0 & 1 \\
    1 & 0 & 0 & 0 \\
    0 & 1 & 0 & 0 \\
    0 & 0 & 1 & 0 \\
    0 & 0 & 0 & 1
\end{array}
\right)
\ \and \
\left(
  \begin{array}{cccc}
    1 & 0 & 0 & 0 \\
    0 & 1 & 0 & 0 \\
    0 & 0 & 1 & 0 \\
    0 & 0 & 0 & 1 \\
    0 & 0 & \frac{\partial f_{21}}{\partial u_3} & \frac{\partial f_{21}}{\partial u_4} \\
    0 & 0 & \frac{\partial f_{22}}{\partial u_3} & \frac{\partial f_{22}}{\partial u_4} \\
    \frac{\partial f_{11}}{\partial u_1} & \frac{\partial f_{11}}{\partial u_2} & 0 & 0 \\
    \frac{\partial f_{12}}{\partial u_1} & \frac{\partial f_{12}}{\partial u_2} & 0 & 0
\end{array}
\right).
$$
Thus, the graph map of $f$ is transversal at the diagonal if and only if
$$
\det\left(
  \begin{array}{cccccccc}
  1 & 0 & 0 & 0 & 1 & 0 & 0 & 0 \\
  0 & 1 & 0 & 0 & 0 & 1 & 0 & 0 \\
  0 & 0 & 1 & 0 & 0 & 0 & 1 & 0 \\
  0 & 0 & 0 & 1 & 0 & 0 & 0 & 1 \\
  1 & 0 & 0 & 0 & 0 & 0 & \frac{\partial f_{21}}{\partial u_3} & \frac{\partial f_{21}}{\partial u_4} \\
  0 & 1 & 0 & 0 & 0 & 0 & \frac{\partial f_{22}}{\partial u_3} & \frac{\partial f_{22}}{\partial u_4} \\
  0 & 0 & 1 & 0 &  \frac{\partial f_{11}}{\partial u_1} & \frac{\partial f_{11}}{\partial u_2} & 0 & 0 \\
  0 & 0 & 0 & 1 & \frac{\partial f_{12}}{\partial u_1} & \frac{\partial f_{12}}{\partial u_2} & 0 & 0
\end{array}
\right) \ne 0,
$$
i.e.,
$$\det\left(
  \begin{array}{cccc}
    1 & 0 & -\frac{\partial f_{21}}{\partial u_3} & -\frac{\partial f_{21}}{\partial u_3} \\
    0 & 1 & -\frac{\partial f_{22}}{\partial u_3} & -\frac{\partial f_{22}}{\partial u_3} \\
    -\frac{\partial f_{11}}{\partial u_1} & -\frac{\partial f_{11}}{\partial u_2} & 1 & 0 \\
    -\frac{\partial f_{12}}{\partial u_1} & -\frac{\partial f_{12}}{\partial u_2} & 0 & 1
\end{array}
\right)
=
\det(\left(
  \begin{array}{cc}
    1 & 0 \\
    0 & 1
 \end{array}
\right)
-
\left(
  \begin{array}{cccc}
    \frac{\partial f_{11}}{\partial u_1} & \frac{\partial f_{11}}{\partial u_2} \\
    \frac{\partial f_{12}}{\partial u_1} & \frac{\partial f_{12}}{\partial u_2}
\end{array}
\right)
\left(
  \begin{array}{cccc}
    \frac{\partial f_{21}}{\partial u_3} & \frac{\partial f_{21}}{\partial u_4} \\
    \frac{\partial f_{22}}{\partial u_3} & \frac{\partial f_{22}}{\partial u_4}
\end{array}
\right))\ne 0.
$$
This is equivalent to say that the map
$$h: \psi_i(P_1)\times \psi_j(Q_1)\to U_i\times U_j,$$
given by
$$(u_1, u_2, u_3, u_4)\mapsto (u_1-f_{21}(u_3, u_4), u_2-f_{22}(u_3, u_4), u_3-f_{11}(u_1, u_2), u_4-f_{12}(u_1, u_2)),$$
has a non-singular differential at the points with $h$-image $(0,0,0,0)$. We shall fit into such a requirement by making $(0,0,0,0)$ to be a regular value.

By Sard theorem, there are real number $c_{11}, c_{12}, c_{21}, c_{22}$ with arbitrary small $c_{11}^2+c_{12}^2+c_{21}^2+c_{22}^2$ such that $(c_{11}, c_{12}, c_{21}, c_{22})$ is a regular value of $h$.  We replace $f_{pq}$ with $f_{pq}+c_{pq}\lambda_{pq}$  for $p,q=1,2$, where $\lambda_{1q}: R^2\to [0,1]$ is smooth maps with  $\lambda_{1q}|_{\psi_i(P_1)}=1$ and the support $Supp(\lambda_{1q})\subset \psi_i(U_i)$, and where $\lambda_{2q}: R^2\to [0,1]$ is smooth maps with  $\lambda_{2q}|_{\psi_j(Q_1)}=1$ and $Supp(\lambda_{2q})\subset \psi_j(U_j)$. We then obtain new homeomorphisms $f_1$ and $f_2$ such that the graph of new alternating homeomorphism is transversal  to diagonal  at $\bar P_1\times \bar Q_1$. Still $\fix f\subset \cup_{k=1}^m P_k\times Q_k$ because our changing is small.

Assume inductively that the graph of new alternating homeomorphism is transversal to diagonal at $\cup_{k=1}^s \bar P_k\times \bar Q_k$, and that $\fix f\subset \cup_{k=1}^m P_k\times Q_k$.  We can change $f_1$ at $P_{s+1}$ and $f_2$ at $Q_{s+1}$ such that   the graph of new alternating homeomorphism is transversal  to diagonal  at $\bar P_{s+1}\times \bar Q_{s+1}$.
Since our changing is arbitrary small and since transversality is stable, the graph of the result map is transversal to diagonal at $\cup_{k=1}^s \bar P_k\times \bar Q_k$.
\end{proof}

\begin{lem}\label{alternating}
If $f: S_1\times S_2\to S_1\times S_2$ is an alternating homeomorphism, then the nature map
$$\rho: S_1\to S_1\times S_2, \quad a\mapsto (a, f_1(a))$$
induces an index-preserving one-to-one corresponding between the set $\fpc(f_2\comp f_1)$ of fixed point classes of $f_2\comp f_1$ and the set $\fpc(f)$ of fixed point classes of $f$.
\end{lem}

\begin{proof}
Note that
$$\fix f=\{(a,b)|f_1(a)=b, f_2(b)=a\}=\{(a, f_1(a))|a\in \fix (f_2\comp f_1)\}.$$
Suppose that $a$ and $a'$ are in the same fixed point class of $f_2\comp f_1$, and $p_i: \widetilde S_i\to S_i$ is the universal cover. Then there is a lifting $\tilde f_1$ of $f_1$ and a lifting $\tilde f_2$ of $f_2$ such that $a,a'\in p_1(\fix (\tilde f_2\comp \tilde f_1))$, there is a point $\tilde a\in p_1^{-1}(a)$ and a point $\tilde a'\in p_1^{-1}(a')$ with $(\tilde f_2\comp \tilde f_1)(\tilde a) = \tilde a$ and $(\tilde f_2\comp \tilde f_1)(\tilde a') = \tilde a'$.
Hence, $(\tau\comp(\tilde f_1\times \tilde f_2)) (\tilde a, \tilde f_1(a)) = ((\tilde f_2\comp\tilde f_1)(\tilde a), \tilde f_1(a))=(\tilde a, \tilde f_1(a))$.
It follows that $(a, f_1(a))\in (p_1\times p_2)(\fix (\tau\comp(\tilde f_1\times \tilde f_2)))$. Similarly, we also have that $(a', f_1(a'))\in (p_1\times p_2)(\fix (\tau\comp(\tilde f_1\times \tilde f_2)))$. Since $\tau\comp(\tilde f_1\times \tilde f_2)$ is a lifting of $f$, we obtain that $(a, f_1(a))$ and $(a', f_1(a'))$ are in the same fixed point class of $f$.
Conversely, suppose that $(a, f_1(a))$ and $(a', f_1(a'))$ are in the same fixed point class of $f$. Then there is a lifting $\tilde f_1$ of $f_1$ and a lifting $\tilde f_2$ of $f_2$ such that both $(a, f_1(a))$ and $(a', f_1(a'))$ lie in $(p_1\times p_2)(\fix (\tau\comp(\tilde f_1\times \tilde f_2)))$. Hence, $a,a'\in p_1(\fix (\tilde f_2\comp \tilde f_1))$, we obtain that $a$ and $a'$ are in the same fixed point class of $f_2\comp f_1$.

Now we shall prove that as a one-to-one correspondence between the sets of fixed point classes, $\rho$ is index-preserving. Since the indices of fixed point classes are invariant under homotopies, by above Lemma \ref{transversal}, we may homotope $f_1$ and $f_2$ such that the graph of $f$ is transversal with the diagonal.

Suppose that the differential $Df_1$ of $f_1$ at $a$ is $M$, and the differential $Df_2$ of $f_2$ at $b=f_1(a)$ is $N$. Then the differential $Df$ of $f$ at $(a,b)$ is $\left(\begin{array}{cc}
0 & N\\
M & 0
\end{array}
\right)$.
Hence the index of $f_2\comp f_1$ at $a$ is
$$\ind(f_2\comp f_1, a) = sgn\det(I_2-NM),$$
and the index of $f$ at $(a,b)$ is
$$\ind(f, (a,b))=sgn \det(I_4-\left(\begin{array}{cc}
0 & N\\
M & 0
\end{array}
\right))=sgn \det\left(\begin{array}{cc}
I_2 & -N\\
-M & I_2
\end{array}
\right)=sgn\det(I_2-NM),$$
where $I_k$ is the identity matrix of order $k$. Therefore,
$$\ind(f_2\comp f_1, a)=\ind(f, (a,f_1(a))).$$
\end{proof}

\begin{cor}\label{alter for N&L}
$$N(f)=N(f_2\comp f_1), \quad L(f)=L(f_2\comp f_1).$$
\end{cor}

\begin{rem}
By the commutativity of the index, we also have $N(f)=N(f_1\comp f_2)$ and $L(f)=L(f_1\comp f_2).$
\end{rem}

Directly following from Lemma \ref{alternating}, Corollary \ref{alter for N&L} and Theorem \ref{fpc for surf.}, we have

\begin{prop}\label{alter case}
If $f: S_1\times S_2\to S_1\times S_2$ is an alternating homeomorphism, then

$\mathrm{(A)}$ For every fixed point class $\F$ of $f$, we have
$$2\chi_1-1\leq \ind(f,\F)\leq 1.$$
Moreover, almost every fixed point class $\F$ of $f$ has index $\geq -1$, in the sense
that
$$\sum_{\ind(f,\F)<-1}\{\ind(f,\F)+1\}\geq 2\chi_1,$$
where the sum is taken over all fixed point classes $\F$ with
$\ind(f,\F)<-1$;

$\mathrm{(B)}$ Let $L(f)$ and $N(f)$ be the Lefschetz number
and the Nielsen number of $f$ respectively. Then
$$|L(f)-\chi_1|\leq N(f)-\chi_1.$$
\end{prop}

\section{Proof of Theorem \ref{main thm}}

In this section, we will complete the proof of Theorem \ref{main thm}.

Firstly, we have the following key lemma by \cite[Proposition 4.4]{ZVW} directly.

\begin{lem}[Zhang-Ventura-Wu, \cite{ZVW}]\label{rectangle hom}
Let $S_1$ and $S_2$ be two connected compact hyperbolic surfaces, and
$$\phi: \pi_1(S_1,a_1)\times \pi_1(S_2, a_2)\to \pi_1(S_1,b_1)\times \pi_1(S_2,b_2)$$
an isomorphism. Then $\phi$ must have one of the following forms:

$(i)$ if $S_1$ and $S_2$ are non-homeomorphic, then $\phi=\phi_1\times\phi_2$, i.e.,
$$\phi(\a,\b)=(\phi_1(\a),\phi_2(\b))$$
where $\phi_i: \pi_1(S_i, a_i)\to \pi_1(S_i, b_i)$ is an isomorphism for $i=1,2$.

$(ii)$ if $S_1$ and $S_2$ are homeomorphic, and hence $\pi_1(S_1,b_1)=\pi_1(S_2, b_2)=\pi_1(S_1)$ by identity, then $\phi=\sigma\comp (\phi_1\times \phi_2)$ where $\phi_1$, $\phi_2$ are two automorphisms of $\pi_1(S_1)=\pi_1(S_2)$, and $\sigma$ is the identity or a transposition. Namely, $\phi$ must have one of the following forms:
$$\quad \phi(\a,\b)=(\phi_1(\a),\phi_2(\b))\quad or \quad \phi(\a,\b)=(\phi_2(\b),\phi_1(\a)).$$
\end{lem}

\begin{prop}\label{standard form} Let $f: S_1\times S_2\to S_1\times S_2$ be a homeomorphism, where $S_i (i=1,2)$ are two connected compact hyperbolic surfaces. Then

$\mathrm{(1)}$ if $S_1$ and $S_2$ are non-homeomorphic, then $f$ can be homotoped to a fiber-preserving homeomorphism $f_1\times f_2$;

$\mathrm{(2)}$ if $S_1$ and $S_2$ are homeomorphic (then we can view $S_1$, $S_2$  as two copies of a compact hyperbolic surface $S$), then $f$ can be homotoped to either a fiber-preserving homeomorphism $f_1\times f_2$ or an alternating homeomorphism $\tau\comp(f_1\times f_2)$, where $\tau$ is a transposition.
\end{prop}

\begin{proof}
Pick a point $(a_1, a_2)\in S_1\times S_2$, and suppose $(b_1,b_2)=f(a_1, a_2)$. Now consider the isomorphism $f_{\pi}$ induced by $f$:
$$f_{\pi}: \pi_1(S_1,a_1)\times \pi_1(S_2, a_2)\to \pi_1(S_1,b_1)\times \pi_1(S_2,b_2), \quad [\gamma]\mapsto [f\comp \gamma]$$
where $[\gamma]$ denotes the loop class of the loop $\gamma: (I, \{0, 1\})\mapsto (S_1\times S_2, (a_1,a_2))$.
Recall that $S_1, S_2$ are both hyperbolic surfaces, then by  Lemma \ref{rectangle hom}, $f_\pi$ must have one of the following forms:

$(i)$ $\quad f_\pi=\phi_1\times\phi_2$, if $S_1$ and $S_2$ are non-homeomorphic;

$(ii)$ $f_\pi=\sigma\comp (\phi_1\times\phi_2)$, if $S_1$ and $S_2$ are homeomorphic.

Note that the isomorphism $f_\pi$ is induced by the homeomorphism $f$, then the isomorphism $\phi_i$ preserves the boundary subgroup structure, that is, for each boundary component $F\subset \partial S_i$, there is a boundary component $F'\subset \partial S_i$ such that $\phi_i(\im (\pi_1(F)\hookrightarrow \pi_1(S_i))$ conjugates to the boundary subgroup $\im(\pi_1(F')\hookrightarrow \pi_1(S_i))$. Hence, by the famous Dehn-Nielsen-Bar theorem for hyperbolic surfaces, the isomorphism $\phi_i$ can be induced by a self-homeomorphism $f_i$ of $S_i$. Therefore, in case (i), $f_\pi$ can be induced by the homeomorphism $f_1\times f_2$, which is a fiber-preserving homeomorphism; and in case (ii), $f_\pi$ can be induced by the homeomorphism $\sigma\comp (f_1\times f_2)$, which is a fiber-preserving homeomorphism (resp. an alternating homeomorphism) when $\sigma$ is the identity (resp. $\sigma$ is a transposition).

Since $S_i (i=1,2)$ is a compact hyperbolic surface, it is a Eilenberg-MacLane space $K(\pi_1(S_i), 1)$, and hence $S_1\times S_2$ is also a $K(\pi_1(S_1\times S_2), 1)$ space. Therefore,  the homeomorphism $f$ is homotopic to $f_1\times f_2$ in case (i), or homotopic to $\sigma \comp (f_1\times f_2)$ in case (ii), which implies the conclusions in Proposition \ref{standard form}.
\end{proof}

\begin{proof}[\textbf{Proof of Theorem \ref{main thm}}]
Since the index of fixed point class, the Lefschetz number and the Nielsen number are all homotopy invariants, we can assume that $f$ is a fiber-preserving homeomorphism or an alternating homeomorphism by Proposition \ref{standard form}. Then Theorem \ref{main thm} holds immediately according to Proposition \ref{fiber case} and Proposition \ref{alter case}.
\end{proof}

For any self-homeomorphism and some special cases of selfmaps of the product of two connected compact hyperbolic surfaces $S_1\times S_2$, we have shown some bounds for fixed points in Theorem \ref{main thm}, Proposition \ref{fiber case}, respectively. Now, there is a nature question:

\begin{question}
Does an analogous bound as in Theorem \ref{main thm} hold for any selfmap of $S_1\times S_2$?
\end{question}


\end{document}